\documentclass[12pt, twoside]{article}
\usepackage{amsmath,amsthm,amssymb}
\usepackage{times}
\usepackage{enumerate}

\usepackage{color,graphicx,wrapfig}

\def\titlerunning#1{\gdef\titrun{#1}}
\makeatletter
\def\author#1{\gdef\autrun{\def\and{\unskip, }#1}\gdef\@author{#1}}
\def\address#1{{\def\and{\\\hspace*{18pt}}\renewcommand{\thefootnote}{}%
\footnote {#1}}%
\markboth{\autrun}{\titrun}}
\makeatother
\def\email#1{e-mail: #1}
\def\subjclass#1{{\renewcommand{\thefootnote}{}%
\footnote{\emph{Mathematics Subject Classification (2010):} #1}}}
\def\keywords#1{\par\medskip
\noindent\textbf{Keywords.} #1}

\newtheorem{theorem}{Theorem}[section]

\newtheorem{lem}[theorem]{Lemma}

\numberwithin{equation}{section}

\newcommand{\R}{{\mathbb R}}
\newcommand{\Z}{{\mathbb Z}}

\newcommand{\e}{\varepsilon}

\newcommand{\supp}{\operatorname{supp}}

\begin{document}

\baselineskip=17pt

\titlerunning{Homogenization of the Dirichlet problem, Pointwise Estimates}

\title{Applications of Fourier analysis in homogenization of Dirichlet problem  \\ I. Pointwise Estimates}

\author{Hayk Aleksanyan
\and
Henrik Shahgholian
\and
Per Sj\"{o}lin}

\date{}

\maketitle

\address{H. Aleksanyan: Department of Mathematics and Mechanics, Yerevan State University, Alex Manoogian 1, Yerevan, Armenia; \email{hayk.aleksanyan@gmail.com}
\and
H. Shahgholian: Department of Mathematics, Royal Institute of Technology,100~44  Stockholm, Sweden; \email{henriksh@kth.se}
\and
P. Sj\"{o}lin: Department of Mathematics, Royal Institute of Technology,100~44  Stockholm, Sweden; \email{pers@math.kth.se}
}



\begin{abstract}
In this paper we prove convergence results for homogenization problem for
solutions of partial differential system with rapidly oscillating Dirichlet data.
Our method is based on analysis of oscillatory integrals. In the uniformly convex and smooth domain, and smooth operator and boundary data, we prove
pointwise convergence results, namely
$$|u_{\e}(x)-u_0 (x)| \leq C_{\kappa} \e^{(d-1)/2}\frac{1}{d(x)^{\kappa} }, \ \forall x\in D, \  \forall \ \kappa>d-1,$$
where $u_{\e}$ and $u_0$ are solutions of respectively oscillating and homogenized Dirichlet problems, and $d(x)$ is the distance of $x$ from the boundary of $D$. As a corollary for all  $1\leq p <\infty$ we obtain $L^p$ convergence rate as well.

\end{abstract}

\section{Introduction and main results}
Homogenization in partial differential equation is a well studied topic, and with a major impact in applications, in particular in material sciences, where impurities of material tend to spread
all over and affect both qualitative and quantitative analysis of the materials properties. In situations as such, one tries to make an approximate averaging of impure quantities, to reduce the cost of any numerical computations. The averaging and homogenization technique is now a well developed tool, with an abundance of literatures on the topic. Most of these literatures, so far, take on the analysis of the bulk equation, that relates to the heterogeneity of the material;  e.g. highly oscillating flux, turbulences, singularly perturbed equations. We refer to the book of Bensoussan-Lions-Papanicolaou \cite{BLP} for background and overview.

The problem of homogenizing the boundary data for Dirichlet problem has long been a challenging task, and there are very few and incomplete results in this direction. We refer to a recent paper of  D. G\'{e}rard-Varet and N. Masmoudi \cite{GM} for an introduction as well as some background in the topic; see also \cite{LS}. The problem of homogenization of boundary data has shown to be an extremely hard problem, with deep connections to ergodic and number theory, and diophantine approximation!

To explain this to some extent, let us (for simplicity) consider the problem of finding a harmonic function
in a (bounded) domain $D$ in $\R^d$ ($d\geq 2$), with an oscillating boundary data $g ( x/ \e)$, where $g$ is a continuous $1$-periodic function in $\R^d$. Such a solution can be expressed by Poisson representation
\begin{equation}\label{poisson}
u_{\e} (x) = \int_{\partial D} P(x,y)g(y/ \e) d\sigma (y) ,
\end{equation}
where $P(x,y)= n(y) A(y)\nabla_y G(x,y)$, with $n(y)$ the outward unit normal to $\partial D$ at $y$, $G$ the Greens function and $A$ the matrix representing the operator;
see (\ref{problem-formulation})-(\ref{matrix}). For the Laplacian/harmonic case $A$ is the identity matrix.

If, instead of harmonic functions, we consider solutions to general operators of divergence type, then  a similar representation is possible through the Poisson kernel of that operator. In particular our method is heavily reliant on such a representation. To our best knowledge, this approach was first introduced in \cite{LS}.

The study of limit behavior of the solution now reduces to study of the integral above. In \cite{LS}, the authors use approximation, along with simple compactness and covering to show that when the surface $\partial D$ does not have flat portions of positive area, with rational normals, then the limit exists and equals $|\partial D|\overline g$. The identification of the limiting boundary value is an easy task, once one can use simple foliation geometry, and whether the so-called scaling surfaces cover the  torus. This identification, nevertheless, becomes quite complicated if the operator is also oscillating, in other words if the matrix $A$ in (\ref{problem-formulation})-(\ref{matrix}) is periodic and oscillating $A(x/ \e)$. We refer to \cite{GM} for such an analysis.

It is our intention in this paper, to analyze the speed of convergence for a (non-oscillating) elliptic system, with an oscillating boundary data. The method presented here, does not seem to extend directly to the case of oscillating equations, and needs modification for such an analysis. However, for oscillating boundary data we are able to obtain reasonably good pointwise estimates which imply $L^p$ estimates of order close to $1/2p$.

Readers may consult several outstanding sources for the theory of homogenization and the references therein: \cite{Allaire}, \cite{Jikov}, \cite{Avellan}, \cite{Chechkin}, \cite{Cioranescu}, \cite{Tartar}.

To state the problem at hand, let us start with fixing some notation, and definitions.

Let $D$ be a bounded and uniformly convex domain in $\R^d$ $(d\geq 2)$, and $\Gamma$ be its boundary. In this paper we study the asymptotic behavior of solutions to the following problem:
\begin{equation}\label{problem-formulation}
\begin{cases} \mathrm{div}(A(x)  \nabla u_{\e}(x))=0,&\text{ in $D$}, \\
u_{\e}(x)=g(x/\e)
,&\text{ on $\Gamma$, } \end{cases}
\end{equation}
where $\e>0$ is a small parameter, $A(y)=(A_{ij}^{ \alpha \beta }(y))$, $1\leq \alpha, \beta \leq d$, $1\leq i,j\leq N$ is an $\R^{N^2 \times d^2}$-valued function defined on $\R^d$, and $g$ be $\mathbb{C}^N$-valued function defined on $\mathbb{T}^d$. Using the summation convention for repeated indices the operator is defined as
\begin{equation}\label{matrix}
\mathcal{L}:=-\mathrm{div} \left[ A(x )\nabla  \right]=- \frac{\partial}{\partial x^{\alpha}} \left[ A^{\alpha \beta }_{ij } ( x )  \frac{\partial }{\partial x^{\beta}}   \right].
\end{equation}

We also consider the corresponding homogenized problem, namely
\begin{equation}\label{problem-formulation-homogen}
\begin{cases} \mathrm{div}(A(x)  \nabla u_0(x))=0,&\text{ in $D$}, \\
u_0(x)=\overline{g}
,&\text{ on $\Gamma$, } \end{cases}
\end{equation}
where $\overline{g}=\int\limits_{\mathbb{T}^d} g(x)dx$. Due to results of \cite{LS} solutions to (\ref{problem-formulation}) converges to
(\ref{problem-formulation-homogen}) under certain geometric conditions on the boundary of $D$, which include the case of strictly convex domains.

\subsection{Assumptions} We make the following assumptions:

\begin{itemize}
    \item[i]   (Periodicity) The boundary value $g$ is $1$-periodic:
$$
g(x+h)=g(x), \qquad \forall x\in \R^d, \ \forall h\in \Z^d.
$$
\item[ii] (Ellipticity) There exists a constant $c>0$ such that
$$
  c^{-1} \xi_{\alpha}^i \xi_{\alpha}^i \leq A_{ij}^{\alpha \beta}(y) \xi_{\alpha}^i \xi_{\beta}^j \leq c \xi_{\alpha}^i \xi_{\alpha}^i, \qquad \forall \xi \in \R^{d\times N},
$$
\item[iii]  (Convexity) We assume that $\Gamma$ is uniformly convex hypersurface, that is all its principal curvatures are bounded away from 0. The constants then depend on the smallest principal curvature.
\item[iv] (Smoothness) For the boundary value $g$ we assume $C^1$ smoothness. Also we suppose that all elements of $A$ and the surface $\Gamma$ are sufficiently smooth. Here we do not aim to obtain the optimal smoothness, but rather focus on the method itself.
\end{itemize}

Our main result is the following theorem.

\begin{theorem} (Pointwise estimate)\label{Thm-pointwise estimate}
    Let $D \subset \R^d$ be a bounded, smooth, and uniformly convex domain.
For each $\kappa  >d-1$ there exists a constant $C_{\kappa}$ such that
\begin{equation}\label{Pointwise estimate}
|u_{\e}(x) -u_0(x) |\leq C_{\kappa} \e^{(d-1)/2}  \frac{1}{d(x)^{\kappa}}, \qquad \forall x\in D
\end{equation}
where $d(x)$ is the distance of $x$ from the boundary of $D$. The constant $C_\kappa$ depends on the smallest principal curvature, and other operator related ingredients.
\end{theorem}

As a consequence of this theorem one can obtain $L^p$-estimates for $1\leq p <\infty$, as follows.

\begin{theorem}($L^p$ estimate) \label{Thm-L_p estimate}
Let $D \subset \R^d$ be as in Theorem \ref{Thm-pointwise estimate}. Then for each $1\leq p<\infty$, and each $\kappa<1/2p$ there exists a constant $C_{\kappa}$ such that
$$|| u_{\e}-u_0 ||_{L^p(D)} \leq C_{\kappa} \e^{\kappa }.$$
\end{theorem}

\vspace{3mm}
\noindent
{\bf Ideas in the proof:} The main idea in the proof is quite simple. Using the Poisson representation as in (\ref{poisson}), in conjunction with a Fourier expansion of the oscillating periodic function $g(x/\e)$ one realizes that all we need to do is to consider this as the Fourier transform of a surface carried measure, with singular weight.
The problem thus reduces to studying the behavior of this Fourier transform. Such problems have been extensively studied in the past, in connection to Hilbert transforms, and singular integrals, see e.g. \cite{PS}. It is nevertheless, not so apparent how one can apply the already existing results and methods to the particular case of our integral. But to stay fair, and not take a complete credit for the methodology presented here, we have to stress that the novelty of our paper is the suggestion of the approach rather than the technique itself. These techniques are fairly standard in Harmonic Analysis, and are based on careful estimates, rather than neat use of geometry, that is more usual in PDE.

We believe that our method can pave the way for further deepening into the subject, at least for the divergence type operators. It surely has the potential of being applied to other areas such as dynamical system and even number theory. But we leave this discussions out here, as the ideas are still very vague, and focus on the main problem.

\section{Preliminaries}

\subsection{Notation} For $\e>0$ set $g_{\e}(x):=g(x/\e)$. For $x\in \R^d$ and $r>0$ by $B(x,r):=\{y \in \R^d: \ |x-y|<r \}$ we denote an open ball with center $x$ and radius $r$.
By $P(x,y)$ we set the vector-valued Poisson kernel for $(\ref{problem-formulation})$.

Here we recall some standard multi-index notation. For a multi-index $\alpha=(\alpha_1,...,\alpha_d)\in \Z^d$, and for a point $x=(x_1,...,x_d)\in \R^d$ we denote $x^{\alpha}:=x_1^{\alpha_1}\cdot ... \cdot x_d^{\alpha_d}$, for $m=(m_1,...,m_d)\in \Z^d$ we set $|m|:=|m_1|+... + |m_d| $. Also by $||m||$ we denote the Euclidian norm of $m$ to avoid ambiguity with its modulus.

In the sequel $C$, $C_1$, $C_2$,... will denote  absolute constants which can vary in different formulas.

\subsection{Preliminary tools}
\begin{lem}\label{Lemma-Poisson-derivatives}
For each $\alpha \in \Z^d_+$ there exists a constant $C_{\alpha}$ such that
\begin{equation}\label{Poisson-estimate_grad-fixed}
|D^{\alpha}_y P(x,y)|\leq C_{\alpha} \frac{1}{|x-y|^{d+|\alpha|-1}}, \qquad x\in D, y\in \Gamma.
\end{equation}
\end{lem}

\begin{proof}
We first consider the case when $d\geq 3$. Let $G(x,y)$ be the Green's function of operator $L:=div(A \nabla)$ in $D$, and let $n(y)$ be the unit exterior normal to $\Gamma$ at $y$. Then the Poisson kernel $P(x,y)$ is defined as
$$
P(x,y)=n(y)A(y)\nabla_y G(x,y), \qquad x\in D, y\in \Gamma,
$$
hence it is enough to prove that
$$
|D^{\alpha}_y G(x,y)|\leq C_{\alpha} \frac{1}{|x-y|^{d+|\alpha|-2}}, \qquad x\in D, y\in \Gamma,
$$
where $C_{\alpha}$ depends on sup norm of $G$, and the operator $\mathcal{L}$.

Fix $x\in D$, $y \in \Gamma$, and set $r=|x-y|$. Consider translated and scaled domain $\widetilde{D}:=\frac 1r (D-x)$, and set $U=(B_2(0) \setminus B_{1/2}(0)) \cap \widetilde{D}$. Since $\widetilde{G}(z):=r^{d-2} G(x,rz+x )$  will be a solution of the scaled problem, and since the domain $U$ and all coefficients of the operator are smooth, then by standard elliptic estimates (see \cite{ADN1}-\cite{ADN2}) we will have
$$
|D^{\alpha}_z \widetilde{G}(z)| \leq C_{\alpha}, \qquad z\in \widetilde{D},
$$
for some constant $C_{\alpha}$ depending on $\alpha$ and independent of $r$. From the latter we obtain
$$
C_{\alpha} \geq |r^{d-2} r^{|\alpha|} D^{\alpha}_y G (x, rz+x)|= r^{d-2+|\alpha| }  | D^{\alpha}_y G (x,y)|,
$$
hence the claim.

The case $d=2$ is treated in analogy with Theorem 13 in \cite{Avellan}.
\end{proof}

%
%

\begin{lem}\label{Lemma-Multiindex-sum}
There exists a constant $c>0$ depending on space dimension only, such that
$$
\sum\limits_{\alpha_1+...+\alpha_d=k} |m^{\alpha} | \geq c ||m||^{k},
$$
where $k\in \mathbb{N}$, $m \in \Z^d$, and $\alpha \in \Z_{+}^d$.
\end{lem}

\begin{proof}
Consider the function
$$
f(x)=\frac{1}{|| x||^{k} } \sum\limits_{\alpha_1+...+\alpha_d=k} |x_1|^{\alpha_1} \cdot ... \cdot |x_d|^{\alpha_d},
$$
where $x\neq 0$. Clearly it is continuous and positive on the unit sphere, hence $c:=\min\limits_{||x ||=1} f(x)>0$. Also observe that $f(tx)=f(x)$ for each $t>0$. Now the statement can be read off.
\end{proof}

\begin{lem}\label{Lemma-coeff-convergence}
If $f\in C^k(\mathbb{T}^d)$ and $\beta \in \mathbb{R}$, then provided $k+\beta> \frac{1}{2}d$, one has
$$
\sum\limits_{m\neq 0, m\in \Z^d} \frac{1}{|| m ||^{\beta}}  |c_m(f)| \leq C_{k+\beta} \left( \sum\limits_{\alpha \in \Z^d_+, |\alpha|=k} ||D^{\alpha} f ||_2^2 \right)^{1/2},
$$
where $c_m(f) = \int\limits_{ \mathbb{T}^d } f(x) e^{-2\pi i m \cdot x} dx$ is the $m$-th Fourier coefficient of $f$.
\end{lem}

\begin{proof}
First observe that using integration by parts we get $c_m(f)=\frac{1}{( 2\pi i m )^{\alpha}} c_m(D^{\alpha}(f))$, for each $m \in \Z^d\setminus \{0\}$ and each $\alpha \in \Z^d_{+}$ with $\alpha_1+...+\alpha_d\leq k$.

Since $D^{\alpha}(f)\in L^2(\mathbb{T}^d)$ from Parseval we have
\begin{equation}\label{L2 norm of derivate}
\sum\limits_{m\neq 0}\sum\limits_{\alpha_1+...+\alpha_d =k} |m^{\alpha} c_m(f)|^2 \leq \sum\limits_{\alpha_1+...+\alpha_d =k} ||D^{\alpha} f||_2^2.
\end{equation}

From Lemma \ref{Lemma-Multiindex-sum}, and  H\"older inequality we obtain
$$
\sum\limits_{m\neq 0} \frac{|c_m(f)|}{|| m||^{\beta}} \leq C \sum\limits_{m\neq 0} \frac{|c_m(f)|}{|| m||^{\beta}} \left(  ||m ||^{-2k}\sum\limits_{\alpha_1+...+\alpha_d=k} |m_1|^{2\alpha_1}\cdot ... \cdot |m_d|^{2\alpha_d}  \right)^{1/2} =
$$
$$
C \sum\limits_{m\neq 0} \left[ |c_m(f)|  \left( \sum\limits_{\alpha_1+...+\alpha_d=k} |m_1|^{2\alpha_1}\cdot ... \cdot |m_d|^{2\alpha_d}  \right)^{1/2} \right] || m||^{-(k+\beta)}\leq
$$

$$
C \left(\sum\limits_{m\neq 0} \sum\limits_{\alpha_1+...+\alpha_d =k }  |m^{\alpha} |^2  |c_m(f)|^2  \right)^{1/2} \left( \sum\limits_{m\neq 0} ||m ||^{-2(k+\beta)}  \right)^{1/2} .
$$
The second factor is finite due to condition $k+\beta>d/2$ of the lemma. The conclusion now follows from $(\ref{L2 norm of derivate})$.
\end{proof}

%
%
Now we formulate and prove our main lemma. Let $\xi=(\xi_1,...,\xi_d) \in \R^d$ $(d\geq 2)$ be fixed, $\xi'=(\xi_1,...,\xi_{d-1})$, $a_0$ and $\rho$ be small positive numbers. Assume that for a function $\psi: \R^{d-1} \rightarrow \R$ we have $\psi \in C^{\infty} (B(0,b)) $, $\psi(0)=0$, $|D^{\alpha} \psi (z') |\leq C_{\alpha}$, and
$$\left|det \left( \frac{\partial^2 \psi }{ \partial z'_i \partial z'_j } \right)_{i,j=1}^{d-1} (z') \right|\geq c>0,$$
for each $\alpha \in \Z^{d-1}_{+}$ and $z' \in B(0, b)$.

Assume further  $u: \R^{d-1} \rightarrow \R$ satisfies $u\in C^{\infty}(B(0,a_0))$, $\supp (u)\subset B(0,\frac 34 a_0)$, and $|D^{\alpha} u(z')|\leq C_{\alpha}$ for each $\alpha \in \Z^{d-1}_{+}$ and $z' \in B(0, a_0)$.
\begin{lem}\label{Lemma-Integral-est}
Under the assumptions and notation above for $(d-1)$-dimensional integral
$$
\mathcal{J}=\int\limits_{ |z'|<a_0 } e^{2\pi i (\xi' \cdot z' + \xi_d \frac{1}{\rho}  \psi (\rho z')  )} u(z') dz'
$$
one has
$$
|\mathcal{J}|\leq C (\rho |\xi|)^{-(d-1)/2}.
$$
\end{lem}

\begin{proof}
Set $\lambda (\eta_1, \eta_2)=(\xi', \rho \xi_d )$ with $\lambda = (|\xi'|^2 +\rho^2 \xi_d^2)^{1/2}$ and $\eta_1^2 +\eta_2^2=1$. Clearly $\lambda \eta_1 =\xi'$ and $\lambda \eta_2 =\rho \xi_d$. Note that since $\lambda \geq (\rho^2 |\xi'|^2 + \rho^2 \xi_d^2)^{1/2} = \rho |\xi|$, it is enough to prove that
\begin{equation}\label{est-J leq lambda}
|\mathcal{J}| \leq C \lambda^{-(d-1)/2}.
\end{equation}
Denote $F( z' ):=\eta_1 \cdot z' + \eta_2 \frac{1}{\rho^2} \psi (\rho z')$. Clearly $\nabla F(z')= \eta_1 +\eta_2 \frac{1}{\rho} \nabla \psi(\rho z')$. Then we have
$$
\mathcal{J}=\int\limits_{ |z'|<a_0 } e^{2\pi i \lambda  F(z')} u(z') dz' .
$$

\noindent \textbf{Case 1.} $\nabla \psi(0) =0 $ and $d=2$.

For $|\eta_2| \leq c_1$, where $c_1$ is a small constant, we have $|F'(z')|>c_2>0$. Here we can invoke integration by parts in $\mathcal{J}$, and using the fact that the derivatives of $u$ are bounded get an estimate $|\mathcal{J}|\leq C \lambda^{-M}$, where $M>0$ is large, and hence also $(\ref{est-J leq lambda})$ for $d=2$.

If $|\eta_2|>c_1$ we get $|F''( z' )|\geq c_2>0$ since $F''( z' )=\eta_2 \psi''(\rho z' )$ and $\psi''(0)\neq 0$. We therefore can apply van der Corput's Lemma (see \cite{Stein}, p. 334) and obtain the estimate $|\mathcal{J}| \leq C \lambda^{-1/2}$.

\noindent \textbf{Case 2.} $\nabla \psi(0) =0 $ and $d>2$.

We first assume that $|\eta_2|\leq c_1$, where $c_1>0$ is a small constant. As in the case $d=2$ we integrate by parts and obtain
$|\mathcal{J}| \leq C \lambda^{-M}$ where $M>0$ is large, and hence $(\ref{est-J leq lambda})$ follows.

It remains to study the case when $|\eta_2| \geq c_1$. We set $Q_{\rho}(z')=-\frac{1}{\rho} \nabla \psi(\rho z')$, and get
$$
\nabla F( z' )=\eta_2 \left( \frac{\eta_1}{\eta_2} + \frac{1}{\rho} \nabla \psi( z' )  \right) =\eta_2 \left( \frac{\eta_1}{\eta_2} - Q_{\rho}( z' )  \right),
$$
hence $z'$ is a critical point of $F$ if and only if $Q_{\rho}( z' )= \frac{\eta_1}{\eta_2}$. Observe, that since the Hessian matrix of $\psi$ is non zero at $z'=0$ it follows that the mapping $z' \mapsto y=Q_{\rho}(z')$ is one-to-one close to the origin for every small $\rho>0$.

We have $F(z')=\eta_2 f(z')$ where
$$f(z') =\frac{\eta_1}{\eta_2} \cdot z' + \frac{1}{\rho^2} \psi(\rho z'),$$

$$ \nabla f(z') = \frac{\eta_1}{\eta_2} +\frac{1}{\rho} \nabla \psi(\rho z')  =  \frac{\eta_1}{\eta_2} -Q_{\rho}(z'), $$
and  $D^2  f (z') = D^2 \psi (\rho z') $, where $D^2$ denotes the Hessian matrix. We have
$$
\mathcal{J} =\int\limits_{|z'|<a_0} e^{2 \pi i \lambda \eta_2 f(z')  } u(z') dz',
$$
where $\supp (u) \subset B(0, \frac 34 a_0)$. Since the determinant of the Hessian matrix of $\psi$ is non zero at $z'=0$ there exists a ball $B(0,R)$ and a positive constant $C_1$ such that for any $x',z' \in B(0,R)$ one has
\begin{equation}\label{est-Lip-G}
\frac 1C_1 |x'-z'| \leq |Q_{\rho}(x') - Q_{\rho} (z')| \leq C_1 |x'-z'|.
\end{equation}
Clearly we may also assume that $|D^2 f(x')|\geq c>0$ for $x' \in B(0,R) $. It follows from $(\ref{est-Lip-G})$ that
$$
\frac 1C_1 |x'-z'| \leq |\nabla f(x') - \nabla f (z')| \leq C_1 |x'-z'|,
$$
and using the fact that $\nabla \psi(0)=0$ we obtain
$$
\frac 1C_1 |z'| \leq |Q_{\rho}(z')| \leq C_1 |z'|,
$$
for each $x',z' \in B(0,R)$. One can see that there exists a neighborhood $\mathcal{M}$ of $0$ such that if $x' \in \mathcal{M}$ then there exists $z' \in \overline{B(0,\frac 14 R)} $ with $Q_{\rho}(z')=x'$. Here the constants $R$, $C_1$, and the neighborhood $\mathcal{M}$ are independent of $\eta$ and $\rho$. Then choose $a_0$ so that $B(0,a_0) \subset  B(0, \frac 14 R)$ and $B(0, 2 C_1 a_0) \subset \mathcal{M}$.

First assume that $\left| \frac{\eta_1}{\eta_2} \right| \geq 2 C_1 a_0 $. We have
$$
|\nabla f (z')| =\left| \frac{\eta_1}{\eta_2} - Q_{\rho}(z') \right| \geq \left| \frac{\eta_1}{\eta_2} \right| - |Q_{\rho}(z') | \geq
$$

$$
2 C_1 a_0 - C_1 |z'| \geq 2 C_1 a_0 - C_1 a_0 = C_1 a_0, \qquad |z' |\leq a_0.
$$
Hence we can integrate by parts in $\mathcal{J}$ and obtain the inequality $(\ref{est-J leq lambda})$. We then assume that $\left| \frac{\eta_1}{\eta_2} \right| < 2 C_1 a_0$.  In this case $\frac{\eta_1}{\eta_2} \in \mathcal{M}$ and there exists $z'_0  \in \overline{B(0,\frac 14 R)} $ such that $Q_{\rho}(z'_0) =\frac{\eta_1}{\eta_2}$, that is $z'_0$ is a critical point for $f$. We have $\nabla f (z'_0)=0$ and therefore
$$
\frac 1C_1 |z'-z'_0| \leq |\nabla f(z')| \leq C_2 |z'-z'_0|  \text{ for } z' \in B(0,R).
$$
Now we can use Theorem 7.7.5 from \cite{H} to obtain the estimate $(\ref{est-J leq lambda})$, and thus completing the proof when $\psi(0)=\nabla \psi(0)=0$.

\noindent \textbf{Case 3.}  $\nabla \psi(0) \neq 0$:

In this case we set $\psi_1(z')= \psi(z')- \nabla \psi (0) \cdot z'$, so that $\psi_1(0)=\nabla \psi_1(0) =0$. Further, for
$$
H(z')=\xi' \cdot z' +\xi_d \frac{1}{\rho} \psi(\rho z').
$$
we have
$$
H(z')= \xi' \cdot z' +\xi_d \frac {1}{\rho} (\psi_1 (\rho z') +\rho \nabla \psi(0) \cdot z' )=
$$
$$
\xi' \cdot z' + \xi_d \nabla \psi(0) \cdot z' +\xi_d \frac{1}{\rho} \psi_1(\rho z' ) =
(\xi' +\xi_d \nabla \psi(0) ) \cdot z' + \xi_d \frac{1}{\rho} \psi_1 (\rho z').
$$

Next setting
$$
 \begin{cases} v'= \xi' + \xi_d \nabla \psi(0)  ,  \\
  v_d=\xi_d.  \end{cases}
$$
or
$$
 \begin{cases} \xi'= v' - v_d \nabla \psi(0)  ,  \\
  \xi_d=v_d,  \end{cases}
$$
with  $c|\xi| \leq |v| \leq C |\xi|$, we shall obtain
$$
H(z')=v' \cdot z' +v_d \frac{1}{\rho} \psi_1 (\rho z').
$$
Since $\nabla \psi_1(0)=0$ we arrive at
$$
|\mathcal{J}|\leq C (\rho |v|)^{-(d-1)/2} \leq C (\rho |\xi|)^{-(d-1)/2},
$$
which completes the proof of the Lemma.
\end{proof}

\section{Proofs of the theorems}

\noindent \textbf{Proof of Theorem \ref{Thm-pointwise estimate}.} Take $r>0$ and consider a covering of $\Gamma$ by the following family of balls $\mathcal{B}=\{B(z, \frac 15 r): \ z\in \Gamma \}$. From covering lemma of Vitali there exists a subfamily of disjoint balls $\mathcal{B}_0 = \{B(z, \frac 15 r): \ j=1,2,...,M \} \subset \mathcal{B}$ so that $\Gamma \subset \bigcup\limits_{j=1}^M B(z_j, r)$. Now fix a non negative $\varphi \in C_0^{\infty}(\R^d)$, such that $\supp (\varphi ) \subset B(0, 2 )$, and $\varphi(x)=1$ for $|x| \leq 1$.  For $r>0$ set $\varphi_r(x)=\varphi( x/r )$, so that $\supp (\varphi_r) \subset B(0, 2 r)$, and for $j=1,2,...,M$ set $\varphi_{r,j}(x)=\varphi_r(x-z_j)$, so $\supp (\varphi_{r,j}) \subset B(z_j, 2 r)$. For $j=1,2,...,M$ denote $B_j:=B(z_j, 2r)$ and $\varphi_j =  \left( \sum\limits_{n=1}^M \varphi_{r,n}(x)  \right)^{-1} \varphi_{r,j}(x) $. It is clear that each function $\varphi_j$ is defined on a neighborhood of $\Gamma$, $\supp (\varphi_j) \subset B_j $, and also $\sum\limits_{j=1}^M \varphi_j(x)=1$, for $x\in \Gamma$. Observe that from the definition of $\varphi_j$ follows that for each $\alpha \in \Z^d$ there exists a constant $C_{\alpha}$ depending on $\alpha$ and the function $\varphi$ such that for all $j=1,2,...,M$ one has

\begin{equation}\label{est-deriv-partofunity}
|D^{\alpha} \varphi_j (x)|\leq C r^{-|\alpha|}, \qquad x\in B_j.
\end{equation}
We also note that an easy volume argument shows that each point of $\Gamma$ can be covered by at most $11^d$ different balls $B_j$.

Now fix a small positive number $a$. Let $x\in D$ and consider a number $0<\rho\leq c_0 d(x)$, where $c_0$ is a small constant, and $d(x)=\mathrm{dist}(x,\Gamma)$. We then take $r=a \rho$, so $\varphi_j$ has support in a ball $B_j$ of radius $2 a \rho$. From $(\ref{est-deriv-partofunity})$ we get
$$
|D^{\alpha} \varphi_j (x)|\leq C_{\alpha} \frac{1}{(a \rho )^{|\alpha|} } =C_{\alpha} \frac{1}{\rho^{|\alpha|}}, \qquad x\in B_j.
$$

We have
$$
u_{\e}(x)-u_0(x) =\int\limits_{\Gamma}P(x,y) [g_{\e}(y)- \overline{g} ] \left( \sum\limits_{j=1}^M \varphi_j(y) \right) d \sigma(y)=
$$

$$
\sum\limits_{j=1}^M \int\limits_{\Gamma_j}P(x,y) [g_{\e}(y)- \overline{g} ] \varphi_j(y) d \sigma(y),
$$
where $\Gamma_j=\Gamma \cap B_j$.

After a permutation of coordinates we may assume that there exists a constant $b>0$ and a smooth, real-valued function $\psi(z')$ defined for $|z'-z_0'|<b$, where $z'=(z_1, ... , z_{d-1})$ such that
$$
\Gamma_j=\{  (z', \psi(z')): \ |z' -z_0' |< 10  a \rho  \} \cap B_j,
$$
where $B_j$ is a ball of radius $2  a \rho$. We may also assume that $\supp \varphi_j \subset B_j$, $|D^{\alpha} \psi (z')|\leq C_{\alpha}$, and $ |det ( \frac{\partial^2 \psi }{\partial z_i \partial z_j}  )_{i,j=1}^{d-1} (z') |\geq c>0$ for $|z'-z_0'|< b$.  The last condition comes from the assumption of the surface being uniformly convex.

Letting $y^{(j)}\in \Gamma_j$ we have
$$
\int\limits_{\Gamma_j} P(x,y) [g_{\e}(y) -\overline{g}] \varphi_j (y) d \sigma (y)=\frac{1}{|x- y^{(j)} |^{d-1}} I_j,
$$
where
$$
I_j=\int\limits_{\Gamma_j} |x-y^{(j)}|^{d-1} P(x,y) [ g_{\e}(y)-\overline{g} ] \varphi_j (y) d \sigma(y).
$$
It follows that
$$
I_j=\int\limits_{|z'- z_0' |<10 a \rho} |x-y^{(j)}|^{d-1} P(x,(z', \psi(z'))) [ g_{\e}(z', \psi(z') )  -\overline{g}  ] \varphi_j (z', \psi(z'))(1+|\nabla \psi(z') |^2 )^{1/2} d z'.
$$
We now make a change of variable $y'=z'-z_0'$ and obtain
\begin{gather*}
I_j=\int\limits_{|y'|<10a \rho  } |x-y^{(j)}|^{d-1} P(x, (z_0' +y', \psi(z_0'+y') ) ) [ g_{\e}( z_0' + y', \psi(z_0' +y'))-\overline{g} ]  \\ \varphi_j (z_0' +y' ,\psi(z_0'+y')) (1+|\nabla \psi(z_0'+y') |^2)^{1/2} dy',
\end{gather*}
We then set $\psi_1(y')=\psi(z_0'+y' ) -s_0$ for $|y'|<b$, where $s_0=\psi(z_0')$. Hence $\psi_1(0)=0$ and
\begin{gather*}
I_j=\int\limits_{|y'|<10  a \rho  } |x-y^{(j)}|^{d-1} P(x, (z_0' +y', s_0+ \psi_1(y') ) ) [ g_{\e}( z_0' + y', s_0+ \psi_1(y'))-\overline{g} ]  \\ \varphi_j (z_0' +y' ,s_0+ \psi_1(y')) (1+|\nabla \psi_1(y') |^2)^{1/2} dy'.
\end{gather*}
We now set $y'= \rho z'$ and obtain
\begin{gather*}
I_j=\rho^{d-1} \int\limits_{|z'| < 10  a  } |x -y^{(j)} |^{d-1} P(x,  ( z_0' +\rho z' , s_0+ \psi_1 ( \rho z' )) ) [ g_{\e}(z_0'+ \rho z', s_0+ \psi_1(\rho z'))-\overline{g} ] \\  \varphi_j (z_0'+\rho z',s_0+ \psi_1( \rho z')) (1+|\nabla \psi_1(\rho z') |^2)^{1/2} d z'.
\end{gather*}
Set $a_0=10 a$. Since $g$ is 1-periodic and smooth we have $g(y)=\sum\limits_{m\in \Z^d} c_m e^{2\pi i m\cdot y}$. It follows that
\begin{gather*}
g_{\e}(z_0'+\rho z', s_0+\psi_1 (\rho z')  )= \sum\limits_{m \in \Z^d} c_m e^{2\pi i  [ \frac{m'}{\e}\cdot (z_0'+\rho z') + \frac{m_d}{\e}(s_0  + \psi_1 (\rho z'))  ] } =\\ \sum\limits_{m\in \Z^d} c_m e^{2\pi i [ \frac{m'}{\e}\cdot z_0' + \frac{m_d}{\e} s_0 ] } e^{2\pi i [ \frac{\rho }{\e} m' \cdot z' + \frac{m_d}{\e} \psi_1(\rho z') ]}
\end{gather*}
Hence
\begin{gather*}
I_j=\rho^{d-1} \sum\limits_{m\neq 0} c_m  e^{2\pi i [ \frac{m'}{\e} \cdot z_0' + \frac{m_d}{\e} s_0 ] }  \int\limits_{|z'|<a_0} |x-y^{(j)}|^{d-1} P(x, (z_0'+  \rho z', s_0+ \psi_1 (\rho z' )) ) \\ \varphi_j (z_0'+ \rho z', s_0+ \psi_1( \rho z'))  ( 1+|\nabla \psi_1 (\rho z')|^2 )^{1/2} e^{2\pi i (\frac{\rho }{\e} m' \cdot z'  +\frac{m_d}{\e} \psi_1(\rho z')   )} dz' = \\
\rho^{d-1} \sum\limits_{m\neq 0} c_m e^{2\pi i [\frac{m'}{\e}\cdot z_0' + \frac{m_d}{\e} s_0 ] }  \int\limits_{|z'|<a_0} e^{2\pi i (\frac{\rho }{\e} m' \cdot z'  +\frac{m_d}{\e} \psi_1(\rho z')   )} u(z') dz',
\end{gather*}
where
$$
u(z')= |x-y^{(j)}|^{d-1} P(x, ( z_0'+ \rho z', s_0+ \psi_1 (\rho z' )) ) \varphi_j (z_0'+ \rho z', s_0 + \psi_1( \rho z')) ( 1+|\nabla \psi_1 (\rho z')|^2 )^{1/2},
$$
and $m'=(m_1,...,m_{d-1})$. It follows from the condition $\rho \leq c_0 d(x)$ and Lemma $\ref{Lemma-Poisson-derivatives}$ that for each $\alpha\in \Z^d_{+}$ there exists a constant $C_{\alpha}$ such that $|D^{\alpha} u(z')|\leq C_{\alpha}$,  for $|z'|<a_0$.

Now by setting $\xi=\frac{\rho}{\e} m$, $\xi' =\frac{\rho}{\e} m'$, and $\xi_d= \frac{\rho}{\e} m_d$ we obtain
\begin{gather*}
I_j=\rho^{d-1} \sum\limits_{m\neq 0} c_m  e^{2\pi i [ \frac{m'}{\e} \cdot z_0' + \frac{m_d}{\e}s_0 ]  }    \int\limits_{|z'|<a_0} e^{2\pi i  ( \xi' \cdot z' + \xi_d \frac{1}{\rho} \psi_1(\rho z') ) } u(z') dz':= \\
\rho^{d-1}\sum\limits_{m\neq 0} c_m e^{2\pi i [ \frac{m'}{\e} \cdot z_0' + \frac{m_d}{\e}s_0 ]  } \mathcal{J}_j,
\end{gather*}
where $\mathcal{J}_j$ denotes the last integral in the sum above. From Lemma $\ref{Lemma-Integral-est}$ we get
$$
| \mathcal{I}_j | \leq C \rho^{d-1} \sum\limits_{m\neq 0} |c_m| (\rho |\xi|)^{-(d-1)/2}= C \e^{(d-1)/2} \sum\limits_{m\neq 0} |c_m| \frac{1}{||m||^{(d-1)/2}} \leq C \e^{(d-1)/2},
$$
where the last sum converges by virtue of Lemma \ref{Lemma-coeff-convergence}. We therefore get
$$
\left| \int\limits_{\Gamma_j}  P(x,y) [g_{\e}(y)-\overline{g} ] \varphi_j(y) d\sigma(y)    \right| \leq C \frac{1}{|x-y^{(j)}|^{d-1} } \e^{(d-1)/2}.
$$

Hence
$$
|u_{\e}(x) - u_0(x)  | \leq C \e^{(d-1)/2 } \sum\limits_{j=1}^M \frac{1}{|x-y^{(j)}|^{d-1} } \leq
$$

$$
C \e^{(d-1)/2 } \rho^{-(d-1)} \sum\limits_{j=1}^M \frac{|\Gamma_j|}{|x-y^{(j)}|^{d-1} } \leq C \e^{(d-1)/2 } \rho^{-(d-1)} \int\limits_{\Gamma} \frac{1}{|x-y|^{d-1}} d \sigma(y) \leq
$$

$$
C \e^{(d-1)/2 } \rho^{-(d-1)} \int\limits_{\Gamma} \frac{1}{|x-y|^{d-1}} \frac{|x-y|^{\delta}}{d(x)^{\delta}}  d \sigma(y) \leq C \e^{(d-1)/2 } \rho^{-(d-1)} \frac{1}{d(x)^{\delta}},
$$
where $\delta>0$ is a small number. Here we used the fact, noted in the beginning of the proof, that for any $y\in \Gamma$ the number of $j$ for which $y\in \Gamma_j$ is bounded by $11^d$.

Now we take $\rho =c_0 d(x)$ where $c_0>0$ is a small constant, and get
$$
|u_{\e}(x) -u_0(x) |\leq C \e^{(d-1)/2}  \frac{1}{d(x)^{d-1+\delta}}, \qquad x\in D,
$$
where $\delta>0$ is arbitrarily small.

Theorem \ref{Thm-pointwise estimate} is proved.

$\newline$
\textbf{Proof of Theorem \ref{Thm-L_p estimate}.} First we consider the case $p=1$. Using Theorem $\ref{Thm-pointwise estimate}$ we get
$$
||u_{\e}-u_0||_{L^1(D)} \leq C \int\limits_{0}^{\e^{1/2}} 1 dt + C \int\limits_{\e^{1/2}}^1 \e^{(d-1)/2} t^{1-d-\delta } dt \leq
$$
\begin{equation}\label{est-p=1}
C \e^{1/2} + C  \e^{(d-1)/2} \e^{1/2(2-d-\delta)} =C \e^{1/2-\delta/2}.
\end{equation}
Now assume $1<p<\infty$. From $(\ref{est-p=1})$ and using the boundedness of $u_{\e}$ and $u_0$ we get
$$
||u_{\e}-u_0||_{L^p(D)} = \left(  \int\limits_{D} |u_{\e}(x)-u_0(x)| |u_{\e}(x)-u_0(x)|^{p-1} dx   \right)^{1/p} \leq
$$

$$
 C \left(  \int\limits_{D} |u_{\e}(x)-u_0(x)|  dx   \right)^{1/p}   \leq C  (\e^{1/2-\delta/2})^{1/p} =
$$

$$
C \e^{1/2p- \delta/2p}.
$$

Theorem \ref{Thm-L_p estimate} is proved. $\square$

\bigskip
\footnotesize
\noindent\textit{Acknowledgments.}
H. Aleksanyan thanks G\"{o}ran Gustafsson foundation for visiting appointment to KTH. H. Shahgholian was partially supported by Swedish Research Council.

\end{document}